\documentclass[12pt,amstex,leqno]{amsart}
\usepackage{latexsym}
\usepackage{amsthm}
\usepackage{amsmath}

\usepackage{amsfonts}
\usepackage{amssymb}
\usepackage{appendix}

\input xy
\xyoption{all}

\swapnumbers

\newtheorem*{rw}{Related Work}
\newtheorem*{ack}{Acknowledgement}

\newtheorem{theorem}{Theorem}[section]
\newtheorem{lemma}[theorem]{Lemma}
\newtheorem{prop}[theorem]{Proposition}
\newtheorem{corollary}[theorem]{Corollary}
\newtheorem{problem}[theorem]{Open problem}

\theoremstyle{definition}
\newtheorem{definition}[theorem]{Definition}

\newtheorem{example}[theorem]{Example}

\newtheorem{exs}[theorem]{Examples}
\newtheorem{remark}[theorem]{Remark}

\numberwithin{equation}{section}

\newcommand\card{\operatorname{card\,}}
\newcommand\Ord{\operatorname{Ord\,}}

\newcommand\ex{\operatorname{ex}}

\newcommand\Ab{\mathbf{Ab}}
\newcommand\Top{\mathbf{Top}}

\newcommand\op{\operatorname{op}}
\newcommand\obj{\operatorname{obj}}

\newcommand\Acc{\operatorname{Acc}}

\newcommand\Set{\mathbf{Set}}
\newcommand\Un{\mathbf{Un}}
\newcommand\Cat{\mathbf{Cat}}
\newcommand\Gra{\operatorname{\mathbf{Gra}}}

\newcommand\colim{\operatorname{colim}}

\newcommand\el{\text{\it el\,}}

\newcommand\ev{\operatorname{\textit{ev}}}

\newcommand\id{\operatorname{id}}

\newcommand\ca{\mathcal {A}} 
\newcommand\cb{\mathcal {B}}
\newcommand\cc{\mathcal {C}}
\newcommand\cd{\mathcal {D}}

\newcommand\cg{\mathcal {G}}

\newcommand\ck{\mathcal {K}}
\newcommand\cl{\mathcal {L}}
\newcommand\crr{\mathcal {R}}

\newcommand\cp{\mathcal {P}}

\newcommand\N{\mathbb N}

\date{April 1, 2019}
\begin{document}

\title{How nice are free completions of categories?}

\author{Ji\v r\'\i{} Ad\'amek}
\address{\newline Department of Mathematics,\newline Faculty of Electrical Engineering,\newline
Czech Technical University in Prague\newline Czech Republic} 
\email{j.adamek@tu-bs.de}

\author{Ji\v r\'\i{} Rosick\'y}
\address{\newline Department of Mathematics and Statistics,\newline
Masaryk University,\newline
 Kotl\'a\v rsk\'a 2, 611 37 Brno,\newline Czech Republic}
\email{rosicky@math.muni.cz}

\thanks{Both authors are supported by the Grant agency of the Czech republic under the grant 19-00902S.
The second author was also supported by the Grant agency of the Czech republic under the grant 
               P201/12/G028.} 

\begin{abstract}
Every category $\ck$ has a free completion $\cp \ck$ under colimits and a free completion $\Sigma\ck$ under coproducts. A number of properties of $\ck$ transfer to $\cp \ck$ and $\Sigma\ck$ (e.g., completeness or cartesian closedness). We prove that $\cp\ck$ is often a pretopos, but, for $\ck$ large, seldom a topos. Moreover, for complete categories $\ck$ we prove that $\cp\ck$ is locally cartesian closed whenever $\ck$ is additive or cartesian closed or dual to an extensive category. We also prove that $\cp \ck$ is (co)wellpowered if $\ck$ is a 'set-like' category, but it is neither wellpowered nor cowellpowered for a number of important categories.
\end{abstract}

\dedicatory{Dedicated to Ale\v s Pultr on his eightieth birthday.}

\maketitle

\section{Introduction}
The free completion $\cp\ck$ of a locally small category $\ck$ under (small) colimits is well known: if $\ck$ is small, then $\cp\ck$ is the presheaf category $[\ck^{\op},\Set]$. For large categories $\cp\ck$ can be described as the full subcategory of 
$[\ck^{\op},\Set]$ on \textit{small} functors, i.e., small colimits of hom-functors. We study several properties of this completion: is it (locally) cartesian closed?, a topos?, wellpowered or cowellpowered? 

For complete categories $\ck$ cartesian closedness of $\cp\ck$ was fully characterized by the second author \cite{R}, but we present
two improvements. One is that cartesian closedeness of $\cp\ck$ implies that $\cp\ck$ is \textit{locally} cartesian closed and complete. And the other is that for $\ck$ complete the completion $\cp\ck$ is locally cartesian closed provided that $\ck$ is
\begin{enumerate}
\item[(a)] additive -- example: $\cp\Ab$, or
\item[(b)] cartesian closed -- examples: $\cp\Set$, $\cp\Cat$, or
\item[(c)] dual to an extensive category -- example: 
$$
\cp(\Set^{\op})=\Acc[\Set,\Set].
$$
\end{enumerate} 
The last example is the category of all accessible set functors. We will see that this category is a locally cartesian closed pretopos,
i.e., a category both exact and extensive. Moreover $\Acc[\Set,\Set]$ is wellpowered, cowellpowered and concrete (i.e., a faithful functor to $\Set$ exists). Another interesting case is the category $\Top$ of topological spaces: the fact, proved by the second author in \cite{R}, that $\cp\Top$ is cartesian closed plays a role in Scott's theory of equilogical spaces \cite{BBS}; we will see that $\cp\Top$ is also a locally cartesian closed pretopos, but it is neither wellpowered nor cowellpowered nor concrete.

Whereas $\cp\ck$ is a topos for every small category, we prove that it practically never is a topos for large ones: if $\ck$ has copowers
and finite intersections, then $\cp\ck$ is not a topos unless $\ck$ is essentially small. However, $\cp\ck$ always is a pretopos
and monomorphisms and epimorphisms in it are regular. 

The situation with the free completion $\Sigma\ck$ under coproducts is simpler. If $\ck$ has a strict initial object 
$0$ (i.e., $\ck(X,0)\neq\emptyset$ inplies $X\cong 0$) and is complete, then $\Sigma\ck$ is cartesian closed iff $\ck$ is. The same holds for local cartesian closedness and for being a topos. However, $\Sigma\ck$ is also cartesian closed for every additive category with products. Here are some examples of the behaviour of $\Sigma\ck$ and $\cp\ck$ for $\ck=\Set$, 
$\Ab$ (abelian groups) or $\Cat$ (small categories), $\ck_f$ denotes the full subcategory of finite sets or groups, resp. We shorten "cartesian closed" to cc and "locally cartesian closed" to lcc.
\vskip 2mm

\renewcommand{\arraystretch}{1.2}
\begin{tabular}{|c|c|c|c|c|c|c|}
\quad & $\ \Set$\  & $\ \Set_f$\ &\  $\Ab$\ &\  $\Ab_f$\ &\   $\Cat$ \quad &\  $\Top$  \quad  \\
\hline
$\ck$&\ \  topos\ \  &\ \ topos\ \  & --- & ---& cc& --- \\
\hline
$\cp\ck$ & lcc &\ \  topos\ \  &\ \ lcc\ \ &\ \  topos\ \  &  lcc & lcc\ \\
\hline
$\Sigma \ck$ & topos & --- & \ cc\,(lcc?)\ \  & cc& --- & ---\\
\hline
\end{tabular}
\vskip 2mm

The connection between the two completions, as proved in \cite{R}, is that $\cp\ck$ is the exact completion of $\Sigma\ck$. 
The  question of when an exact completion is a topos was studied in \cite{M}, at the end of Section 4 we turn to the results of that
paper.

Finally, for "set-like" categories we prove that $\cp\ck$ is wellpowered and cowellpowered. 
 Whereas $\cp \ck$ is neither wellpowered nor cowellpowered for
 $\Ab, \Cat, \Top$ and a number of other categories.
 
\begin{rw}
{
\em
Properties of the category $\cp(\Set^{\op})$ of accessible set functors  are also studied by Barto \cite{B}: he proves that this category
is concrete (using a method different from ours) and universal, i.e., all concrete categories can be fully embedded into $\cp(\Set^{\op})$.

Cartesian closedness of exact completions is investigated by Clementino et al. \cite{CHR}. For example, the exact completion
of metric spaces (in Lawevere's sense) and Lowen's approach spaces are cartesian closed.
}
\end{rw}  

\begin{ack}
{
\em
The authors are very grateful to the referee and the editor whose suggestions have improved the presentation of our results.
}
\end{ack}

\section{The Free Coproduct Completion}
For every category $\ck$ the free completion under coproducts is denoted by 
$$
\Sigma \ck.
$$
That is, the category $\Sigma\ck$ has coproducts, and there is a full embedding $E\colon \ck \to \Sigma \ck$ such that every functor 
$F\colon \ck \to \cl$ where $\cl$ has coproducts has an extension to a coproduct-preserving functor $\overline F \colon \Sigma \ck \to \cl$ unique up to natural isomorphism. This category $\Sigma\ck$ can be described as the category of all collections $A = (A_i)_{i\in I}$ of objects of $\ck$, where a morphism $f\colon (A_i)_{i\in I}\to (B_j)_{j\in J}$ consists of 

(a) a function $\hat f \colon I\to J$, and

(b) a collection of morphisms $f_i\colon A_i \to B_{\hat f(i)}$, $i\in I$.

\noindent
Composition and identity morphisms  are as expected.
The embedding $\ck \hookrightarrow \Sigma \ck$ is, by abuse of notation, denoted by $K\mapsto K$ and $f\mapsto f$.

In the present section, properties of $\Sigma\ck$ are studied. Some of them are well known or easy to see:


\begin{example}\label{sigma0}
(a) $\Set$ is the free coproduct completion of the terminal (one-arrow) category.

(b) The category of sets with a unary relation is the free coproduct completion of the two-element chain $0<1$. Indeed, we have an equivalence functor assigning to every collection $(A_i)_{i\in I}$ of binary values the set $I$ with the relation $\{i\in I; A_i=1\}$.

(c) The category $\Set^{\to}$ is the free coproduct completion of $\Set$.

(d) More generally, if a category $\cc$ with coproducts has the property that every object is a coproduct of coproduct-indecomposable objects, then $\cc=\Sigma \ck$ for the full subcategory $\ck$ on all indecomposable objects. See \cite{CV}, Lemma 42.

(e) $\Sigma\ck$ can be described as the full subcategory of $[\ck^{\op},\Set]$ consisting of coproducts of hom-functors.
\end{example}

\begin{remark}\label{sigma}
(1) $\Sigma \ck$ has (finite) limits iff $\ck$ has multi-limits of all (finite) diagrams, see \cite{HT}. Recall that a multi-limit of a diagram $D$ is a set of cones such that every cone of $D$ factorizes through a unique member of that set and, moreover, the factorization is unique. In categories with an initial object limits and multi-limits are clearly the same.

(2) The embedding $\ck \hookrightarrow \Sigma \ck$ clearly preserves existing limits.

(3) Monomorphisms of $\Sigma \ck$ are precisely the morphisms $f\colon A\to B$ such that $\hat f$ is monic in $\Set$ and each $f_i \colon A_i \to B_{\hat f(i)}$ is monic in $\ck$. We always represent them by collections with $I\subseteq J$ where $\hat f$ is the inclusion map.

(4) Analogously, epimorphisms in $\Sigma \ck$ are precisely the morphisms $f$ with $\hat f$ epic in $\Set$ and each $f_i$ epic in $\ck$.

(5) $\Sigma\ck$ is wellpowered (or cowellpowered) iff $\ck$ is.
\end{remark}

\begin{remark}\label{pullback}
If $\ck$ has pullbacks, then pullbacks of morphisms
$$
A\xrightarrow{\quad f \quad} C \xleftarrow{\quad g\quad} B
$$
in $\Sigma\ck$ are computed as follows: form a pullback of $\hat f$ and $\hat g$ in $\Set$
$$
\xymatrix@C=3pc@R=3pc{
& J \ar[dl]_{\hat p} \ar[dr]^{\hat q}&\\
I_A\ar[dr]_{\hat f} & & I_B \ar[dl]^{\hat g}\\
& I_C & }
$$
and for each $j\in J$ form a pullback in $\ck$:
$$
\xymatrix@C=3pc@R=3pc{
& D_j\ar[dl]_{ p_j} \ar[dr]^{q_j}&\\
A_{p(j)}\ar[dr]_{f \hat p(j)} & & B_{q(j)} \ar[dl]^{g_{\hat q(j)}}\\
& C_{\hat f\hat p(j)} & }
$$
This defines an object $D=(D_j)_{j\in J}$ and morphisms $p\colon D\to A$, $q\colon D\to B$ of $\Sigma \ck$ (by the given functions $\hat p$, $\hat q$ and the given components $p_j$, $q_j$). It is easy to see that the following square
$$
\xymatrix@C=3pc@R=3pc{
& D\ar[dl]_{ p} \ar[dr]^{q}&\\
A\ar[dr]_{f} & & B \ar[dl]^{g}\\
& C & }
$$
is a pullback in $\Sigma \ck$.

\end{remark}

\begin{remark}\label{extensive0}
Recall that a finitely complete category with coproducts is \textit{infinitary extensive} if coproducts are universal (i.e., preserved by pulling back along a morphism) and disjoint (i.e., coproduct injections are monic and their pairwise intersections are formed by the initial object). 

It is \textit{extensive} if finite coproducts are universal and disjoint.
\end{remark} 

\begin{prop}\label{extensive}
For every category $\ck$ the following conditions are equivalent:
\begin{enumerate}
\item[(i)] $\Sigma\ck$ is extensive,
\item[(ii)] $\Sigma\ck$ is infinitary extensive and
\item[(iii)] $\ck$ is finitely multi-complete.
\end{enumerate}
 \end{prop}
\begin{proof}
i $\to$ iii is clear: extensivity includes finite completeness, so use Remark~\ref{sigma}(1).

\noindent
iii $\to$ ii. We prove that coproducts in $\Sigma \ck$ are disjoint and universal. A coproduct of two objects $A=(A_i)_{i\in I}$ and $B= (B_j)_{j\in J}$ is obviously  the collection $C= (C_k)_{k\in I+ J}$ with $C_i=A_i$ and $C_j=B_j$. The description of pullbacks in Remark~\ref{extensive} makes it clear that the pullback of the coproduct injections $i_1\colon (A_i) \to (C_k)$ and $i_2\colon (B_j) \to (C_k)$ is the empty collection, i.e., the initial object of $\Sigma \ck$. The proof that, more generally, coproducts in $\Sigma \ck$  have pairwise disjoint injections is completely analogous.

For the verification that coproducts are universal in $\Sigma \ck$ we also use just a coproduct $(C_k)= (A_i)+ (B_j)$ of a pair of objects, the general proof is again completely analogous. Let
$$ 
f\colon D= (D_l)_{l\in L} \to (C_k)_{k\in I+ J}
$$
be an arbitrary morphism. Then $L= L^1 + L^2$ for $L^1 = \hat f^{-1} (I)$ and $L^2 = \hat f^{-1} (J)$. Thus
$$
D= D^1 + D^2
$$
where  $D^1 = (D_l)_{l\in L^1}$ and $D^2 =(D_l)_{l\in L^2}$. And  the components of $f$ define obvious morphisms $f^1 \colon D^1 \to A$ and $f^2 \colon D^2 \to B$. It is easy to see that we get the following pullbacks
$$
\xymatrix@C=3pc@R=3pc{
D^1\ar [d]_{f^1} \ar [r]^{i_l'}& D
\ar [d]_{f} & D^2 \ar [l]_{i_2'} \ar[d]^{f^2}\\
A\ar[r]_{i_l} & C & B \ar [l]^{i_2}
}
$$
where $i_l'$, $i_2'$ are the coproduct  injections of $D= D^1 +D^2$.
\end{proof}

\noindent
ii $\to$ i is trivial.


\begin{lemma}\label{mproduct}
Let $(A_{ij})_{j\in J}$ be objects of $\Sigma\ck$ for $i\in I$. A product of this family exists in $\Sigma\ck$ iff for every function $\varphi:I\to J$ a multi-product of $A_{i\varphi(i)}$ for $i\in I$ exists in $\ck$.
\end{lemma}
\begin{proof}
(1) We prove the necessity: Let $(C^k)_{k\in K}$ be a product $\prod\limits_{i\in I} (A_{ij})_{j\in J}$ in $\Sigma\ck$ with projections
$$
\pi_i:(C^k)_{k\in K}\to (A_{ij})_{j\in J}
$$
for $i\in I$. Given $\varphi:I\to J$ denote by $K_\varphi\subseteq K$ the set of all $k$ with $\varphi(i)=\hat{\pi}_i(k)$ for all $i\in I$.
This defines cones in $\ck$ as follows: for every $k\in K_\varphi$ we have the $k$-component of $\pi_i$
$$
\pi^k_i:C^k\to A_{i\varphi(i)}\quad\quad (i\in I).
$$
The set of these cones (indexed by $K_\varphi$) is a multi-product of $A_{i\varphi(i)}$ for ${i\in I}$.

Indeed, every cone in $\ck$
$$
f_i:X\to A_{i\varphi(i)}\quad\quad (i\in I)
$$
yields a cone $g_i:X\to (A_{ij})_{j\in J}$ in $\Sigma\ck$ with $\hat{g}_i$ choosing $\varphi(i)$ and the unique component of $g_i$ being $f_i$. The fact that $(g_i)_{i\in I}$ has a unique factorization through $(\pi_i)_{i\in I}$ in $\Sigma\ck$ clearly implies that there exists a unique $k\in K_\varphi$ and a unique factorization of $(f_i)_{i\in I}$ through $(\pi^k_i)_{i\in I}$ in $\ck$.

(2) Conversely, if all the above multi-products exist in $\ck$, then $\prod\limits_{i\in I} (A_{ij})_{j\in J}$ is given in $\Sigma\ck$ as follows. For every $\varphi:I\to J$ choose a multi-product, indexed by a set $K_\varphi$, as follows:
$$
\pi^k_i:C^k\to A_{i\varphi(i)}\quad\quad (i\in I,k\in K_\varphi).
$$
This defines a morphism of $\Sigma\ck$ for every $i\in I$: put $K=\coprod\limits_{\varphi:I\to J} K_\varphi$ and let
$$
\pi_i:(C^k)_{k\in K}\to (A_{ij})_{j\in J}
$$ 
be defined by
$$
\hat{\pi}_i:K\to J \text{ and }\hat{\pi}_i(k)=\varphi(i)
$$
for all $k\in K_\varphi$ with components
$$
\pi^k_i:C^k\to A_{i\varphi(i)}$$
as above.

For every object $D$ of $\ck$ and every cone $f_i:D\to (A_{ij})_{j\in J}$, $i\in I$, in $\Sigma\ck$ we prove that there is a unique factorization through $(\pi_i)_{i\in I}$. It then follows immediately that the same holds for families of more that one object as well.

The given cone chooses for every $i\in I$ a unique $j\in J$, let $\varphi:I\to J$ denote the resulting function. And $f_i$ has a component $g_i:D\to A_{i\varphi(i)}$ in $\ck$, $i\in I$. There exists a unique $k\in K_\varphi$ and a unique factorization of $(g_i)_{i\in I}$ through $(\pi^k_i)_{k\in K_\varphi}$ in $\ck$. This clearly implies that $(f_i)$ factorizes uniquely through $(\pi_i)$ in $\Sigma\ck$.
\end{proof}


\begin{remark}\label{multi0}
Recall that a functor $F:\ca\to\cb$ is a \textit{left multi-adjoint} if, for every object $B\in\cb$, the category $F/ B$ has a multi-terminal set of objects. This is a set $f_i:FA_i\to B$, $i\in I$ such that for every $f:FA\to B$ there is a unique $i\in I$ and a morphism
$g: A\to A_i$ such that $f=f_i\cdot F(g)$, and moreover $g$ is also unique.
\end{remark}

\begin{definition}\label{multi}
A category $\ck$ is called \textit{cartesian multi-closed} if it has finite products and each endofunctor  $A\times -$ is a left 
multi-adjoint.
\end{definition}

Explicitly, this means that for every pair of objects $A$ and $B$ there exists a set of objects $[A,B]_i$ $(i\in I)$ and morphisms
$$
\ev_i \colon A \times [A,B]_i \to B \quad (i\in I)
$$
multi-universal in the expected sense:
For every morphism $c\colon A \times C \to B$ there exists a unique $i\in I$ for which some morphism $\bar c \colon C \to [A,B]_i$ makes the following triangle
 $$
\xymatrix@C=4pc@R=4pc{
A\times  C \ar[d]_{A\times \bar c} \ar [r]^{c} & B\\
A\times [A,B]_i \ar[ur]_{\ev_i}&
}
$$
commutative, and moreover, $\bar c$ is also unique. 

\begin{exs}\label{multi1e}
The following categories are cartesian multi-closed:

(a) Every cartesian closed category. 

(b) Every category with binary products satisfying $A\times B\cong A+B$ (e.g. every additive category). Indeed, given $A,B$, the family $[f,\id_B]:A\times B\cong A+B\to B$ indexed by all $f:A\to B$ is obviously multi-universal: every morphism $c:A+C\to B$ has the form $c=[f,\bar{c}]$ for a unique pair $f:A\to B$ and $\bar{c}:C\to B$.

(c) Every dual of an extensive category (e.g. $\ck=\Set^{\op}$). Indeed, first observe that every object $B$ has only a set of decompositions $B = B_1 + B_2$ in $\ck^{\op}$ (up to isomorphism). In fact, every such decomposition yields a morphism
in $\ck^{\op}(B,B+B)$ taking $B_1$ to the left-hand summand and $B_2$ to the right-hand one. And different decompositions correspond to different morphisms.

Now, given $A,B$, the family $f+\id_{B_2}:B\to A+B_2$ in $\ck^{\op}$ indexed by all decompositions $B=B_1 + B_2$ and all $f:B_1\to A$ is obviously multi-universal: every morphism $c:B\to A+C$ in $\ck^{\op}$ has the form $c=[f,\bar{c}]$ in $\ck$ for unique $f:B_1\to A$ and 
$\bar{c}:B_2\to C$.
\end{exs}

\begin{remark}\label{multi1}
Let $\ck$ have a \textit{strict initial object} $0$, i.e., if a morphism $X\to 0$ exists, then $X$ is initial. Then $\ck$ is cartesian closed iff it is cartesian 
multi-closed. Indeed, consider the above triangle for $C=0$: since $A\times C\cong 0$, we have a unique $c$, but for every $i\in I$ the unique $\bar{c}:0\to [A,B]_i$ makes the triangle commutative. Thus, $I$ is a singleton set.

This implies that for example the category $\Top$ of topological spaces is not cartesian multi-closed.
\end{remark}

\begin{theorem}\label{multi2}
Let $\ck$ have finite products. Then $\Sigma \ck$ is cartesian closed iff $\ck$ is cartesian multi-closed and has multi-products.
\end{theorem}

\begin{proof}
(1) Let $\Sigma \ck$ be cartesian closed. Given objects $A$, $B$ of $\ck$ we denote the corresponding exponential object in $\Sigma \ck$ by
$$
[A,B] = ([A,B]_i)_{i\in I}
$$
together with the counit having components
$$
\ev_i: A\times [A,B]_i  \to B \quad (i\in I)\,.
$$

(1a) We prove that $\ck$ is cartesian multi-closed.
It is our task to prove that for every morphism $c\colon A \times C \to B$ in $\ck$ there exists a unique $i\in I$ with a commutative triangle
 $$
\xymatrix@C=4pc@R=4pc{
A\times  C \ar[d]_{A\times \bar c} \ar [r]^{c} & B\\
A\times [A,B]_i \ar[ur]_{\ev_i}&
}
$$
for some $\bar c \colon C\to [A,B]_i$, and moreover $\bar c$ is unique. Indeed, $c\colon A \times C\to B$ is a morphism of $\Sigma \ck$, thus we have a unique $\tilde c \colon C \to ([A,B]_i)_{i\in I}$ with
$$
c=(\ev_i)_{i\in I}  \cdot (A \times \tilde c)
$$
in $\Sigma \ck$. This morphism $\tilde c$ is given by choosing $i_0 \in I$  and a morphism $\bar c \colon C \to [A,B]_{i_0}$ in $\ck$ with $c =\ev_{i_0}\cdot (A\times \bar c)$.

Conversely, suppose the triangle  above commutes, then we must prove $i=i_0$ and $\bar c$ is as above. The choice of $i$  and $\bar c$ in that triangle defines a morphism $\tilde c\colon C \to ([A,B]_i)_{i\in I}$ with $c =\ev_i\cdot (A\times \tilde  c)$, and since $\tilde c$ is unique, the proof of uniqueness of $i$ and $\bar c$ follows. 

(1b) $\ck$ has multi-products. Indeed, for every collection $A_i$ $(i\in I)$ in $\ck$ let $B=(1)_{i\in I}$ be the object of $\Sigma\ck$ with components $1$, terminal in $\ck$. Since in $\Sigma\ck$ we have $B=\coprod\limits_{i\in I} 1$ and $\Sigma\ck$ has the exponential object $[B,(A_j)_{j\in I}]$, it has the following product
$$
\prod\limits_{i\in I} [1,(A_j)_{j\in I}]=\prod\limits_{i\in I} (A_j)_{j\in I}.
$$
By Lemma \ref{mproduct} for every $\varphi:I\to I$ a multi-product of $A_{\varphi(i)}$ $(i\in I)$ exists in $\ck$. This proves our claim: put $\varphi=\id_I$.

(2) Let $\ck$ be cartesian multi-closed and have multi-products. Then $\ck$ has for every pair $A$, $B$ of objects a multi-universal collection of morphism $\ev^i \colon A \times [A,B]^i \to B$, $(i\in I)$ w.r.t $A\times -$. We first prove that $\Sigma\ck$ has exponential objects $[A,B]$ for arbitrary objects 
$A\in \ck$ and $(B_j)_{j\in J}\in \Sigma \ck$. The exponential object $[A, (B_j)_{j\in J}]$ in $\Sigma \ck$ is formed as follows: let for every $j\in J$ a multi-universal collection be given
$$
\ev_j^i \colon A \times [A,B]_j^i \to B_j \quad (i\in I_j)\,.
$$
This defines a morphism
$$
\ev \colon A \times ([A,B]_j^i)_{j\in J, i\in I_j} \to (B_j)_{j\in J}
$$
whose index function $\hat{ev}$ is $(j,i)\mapsto j$ and whose components are 
$\ev_j^i$. We prove that in $\Sigma \ck$ we have
$$
[ A, (B_j)_{j\in J} ] = ([A,B]_j^i)_{j\in J, i\in I_j}
$$
with the above counit $\ev$. Indeed, to give a morphism $h\colon (D_k)_{k\in K} \to ([A,B]_j^i)$ means to choose for every 
$k\in K$

(a) indices $j\in J$ and $i\in I_j$, and

(b) a morphism $h_k \colon D_k \to [A,B]_j^i$\,.

And to give  a morphism $\bar h \colon A \times (D_k)_{k\in K} \to (B_j) _{j\in J}$ means to choose for every $k\in K$ an index $j\in J$ and  a morphism $\bar h_k \colon A \times  D_k\to B_j$, where the latter is  equivalent to choosing $i\in I_j$ and a morphism from $D_k$ to $[A,B]_j^i$, due to the universal property of $\ev_j^i$.

Turning to a general pair of objects $A=(A_l)_{l\in L}$ and $(B_j)_{j\in J}$, the existence of exponential objects $[A_l, (B_j)_{j\in J}]$ and the fact that $(A_l)$ is a coproduct in $\Sigma\ck$ of $A_l$ imply
$$
[(A_l)_{l\in L}, (B_j)_{j\in J}] = \prod_{l\in L} [ A_l, (B_j)_{j\in J}]\,;
$$
the above product exists because $\ck$ has multi-products (see Lemma \ref{mproduct}).
\end{proof}




\begin{corollary}\label{sigma1}
$\Sigma\ck$ is cartesian closed for every complete, additive category.    
\end{corollary}

\begin{definition}\label{cmc}
A category $\ck$ is called \textit{locally cartesian multi-closed} if every slice $\ck /A$ is cartesian multi-closed.
\end{definition} 

\begin{theorem}\label{sigma2a}
Let $\ck$ be a complete category. Then $\Sigma\ck$ is locally cartesian closed iff $\ck$ is locally cartesian multi-closed.
\end{theorem}
\begin{proof}
Let $\ck$ be locally cartesian multi-closed. For every object $A=(A_i)_{i\in I}$ of $\Sigma\ck$ we verify that $(\Sigma\ck)/A$ is cartesian closed. This category is isomorphic to the product of categories $\Sigma(\ck/A_i)$, $i\in I$. Indeed, consider the functor which to every
object $(B_j)_{j\in J}\to A$ of $(\Sigma\ck)/A$ assigns the $I$-tuple, where for $i\in I$ the object of $\ck/A_i$ is given as follows: the indexing set is $\hat{b}^{-1}(i)\subseteq J$, and for every $j\in\hat{b}^{-1}(i)$ the corresponding morphism is $b_j:B_j\to A_i$.

Since $\ck/A_i$ has products, from its cartesian multi-closedness follows that $\Sigma(\ck/A_i)$ is cartesian closed (see Theorem \ref{multi2}). And a product of cartesian closed categories is cartesian closed (with exponential objects defined componentwise). 

Conversely, let $\Sigma\ck$ be locally cartesian closed. For every object $A$ of $\ck$ the category $\Sigma (\ck/A)$ is isomorphic to 
$(\Sigma \ck)/A$ and is thus cartesian closed. Therefore $\ck /A$ is cartesian multi-closed (see Theorem \ref{multi2}), proving that 
$\ck$ is locally cartesian multi-closed.
\end{proof}

\begin{corollary}\label{sigma2}
Let $\ck$ be a complete category with a strict initial object. Then $\Sigma\ck$ is (locally) cartesian closed if and only if $\ck$ is
(locally) cartesian closed.
\end{corollary}
This follows from Remark \ref{multi1}, Theorem \ref{multi2} and Theorem \ref{sigma2a}. 


\begin{exs}\label{sigma4}
(i) For the topos $\Set_f$ of finite sets, $\Sigma\Set_f$ is not cartesian closed: $\Set_f$ does not have multi-products (see Remark \ref{sigma}(1)). Analogously, $\Sigma\Ab_f$ is not cartesian closed where $\Ab_f$ is the category of finite abelian groups.

(ii) In contrast, $\Sigma\Set$ is locally cartesian closed, indeed a topos, see Example \ref{sigma0}(c). 

(iii) $\Sigma\Ab$ is cartesian closed by Corollary \ref{sigma1}.

(iv) The category $\Cat$ of small categories is cartesian closed but not locally so. Since it is complete and has a strict initial object, we conclude that $\Sigma\Cat$ is cartesian closed but not locally so.
\end{exs}

\begin{problem}\label{sigma5} 
{
\em
Is $\Sigma\Ab$ locally cartesian closed?
}
\end{problem}

\begin{definition}\label{multi4}
By a \textit{multi-topos} is meant a finitely complete, cartesian multi-closed category with a subobject classifier.
\end{definition}

\begin{theorem}\label{multi5}
For a complete category $\ck$ the completion $\Sigma\ck$ is a topos iff $\ck$ is a multi-topos.
\end{theorem}

\begin{proof}
(1) Let $\ck$ be a multi-topos with a subobject classifier $\varepsilon \colon 1 \to \Omega$. Then the morphism
$$
\bar{\varepsilon} \colon 1 \to (\Omega , 1) \quad \mbox{in} \quad \Sigma \ck
$$
corresponding to $\varepsilon$  is a subobject classifier of $\Sigma \ck$, which by Theorem \ref{multi2} proves that $\Sigma\ck$ is a topos.

Indeed, given a subobject $f\colon (A_i)_{i\in I} \to (B_j)_{j\in J}$ with $I\subseteq J$ (see Remark~\ref{sigma}(2)), the corresponding pullbacks in $\ck$

$$
\xymatrix@C=3pc@R=3pc{
A_{i_J}\ar[d]_{!} \ar [r]^{f_i} & B_i\ar[d]^{g_i}\\
1 \ar[r]_{\varepsilon}& \Omega
}\quad{(i\in I)}  
$$ 
yield the following morphism
$$
g\colon (B_j)_{j\in J} \to (\Omega, 1).
$$
The function $\hat g$ assign to every $i\in I$ the left-hand object of $(\Omega,1)$, (with the component $g_i$ above)
and to every $j\in J-I$ the right-hand one. The square below
\begin{equation}\label{ee5.1}
 \xymatrix@C=3pc@R=3pc{
A\ar[d]_{!} \ar [r]^{f} & B\ar[d]^{g}\\
1 \ar[r]_{\varepsilon}& (\Omega, 1)
}
\end{equation}
is a pullback in $\Sigma\ck$. This follows easily from Remark \ref{pullback}.


It remains to prove that the pulback  \eqref{ee5.1} determines $g$ uniquely. Suppose $g' \colon B\to (\Omega, 1)$ makes the corresponding square a pullback in $\Sigma \ck$, too. For every $j\in J$ the commutativity of that square proves that
$\hat{g'}(j)$ is the left-hand object, $\Omega$, iff $j\in J$. Thus, $\hat{g}=\hat{g'}$. For every $i\in I$ we are to prove that the following square
\begin{equation}\label{ee5.2}
\xymatrix@C=3pc@R=3pc{
A_i\ar[d]_{!} \ar [r]^{f_i} & B_i\ar[d]^{g'_i}\\
1 \ar[r]_{ \varepsilon}& \Omega
}
\end{equation}
is a pullback in $\ck$ -- then $g_i=g'_i$ and the proof is complete. Indeed, given a commutative square in $\ck$
$$
\xymatrix@C=3pc@R=3pc{
D\ar[d]_{!} \ar [r]^{d} & B_i\ar[d]^{g'_i}\\
1 \ar[r]_{ \varepsilon}& \Omega
}
$$
define $f^\ast \colon A\to B$ by $\hat f^\ast =\hat f$, $f^\ast_i =d$ and $f^\ast_{i'} = f_{i'}$ for all $i'\in I -\{i\}$. 
Then the last square and the pullback \eqref{ee5.1} with $g'$ in place of $g$ imply that the following square
$$
\xymatrix@C=3pc@R=3pc{
A\ar[d]_{!} \ar [r]^{f^\ast} & B\ar[d]^{g'}\\
1 \ar[r]_{ \bar\varepsilon}& (\Omega, 1)
}
$$
commutes in $\Sigma \ck$. Thus, it factorizes uniquely through the modified square \eqref{ee5.1}, which proves that $d$ factorizes uniquely through \eqref{ee5.2}, as required.

(2) Let $\Sigma \ck$ be a topos. By Theorem \ref{multi2},
\ $\ck$ is cartesian multi-closed. Denote by $\bar\varepsilon \colon 1 \to (\Omega_t)_{t\in T}$ the subobject classifier of $\Sigma\ck$.
Let $t_0 \in T$ be the element given by the indexing function of $\hat \varepsilon$, and let
$$ \varepsilon \colon 1 \to \Omega_{t_0}
$$
be the unique component of $\bar \varepsilon$. For every subobject $f\colon A\to B$ in $\ck$, since $f$ is also monic in $\Sigma \ck$, we have a morphism $g\colon B\to (\Omega_t)_{t\in T}$ forming  a pullback as follows:
$$
\xymatrix@C=3pc@R=3pc{
A\ar[d]_{!} \ar [r]^{f} & B\ar[d]^{g}\\
1 \ar[r]_{ \varepsilon}& (\Omega_t)_{t\in T}
}
$$
Thus $\hat g$ chooses $t_0$ (since $\widehat{\varepsilon \cdot !}$ does) and the unique component $g_0 \colon B\to \Omega_{t_0}$ defines a pullback in $\ck$ as follows:
$$
\xymatrix@C=3pc@R=3pc{
A\ar[d]_{!} \ar [r]^{f} & B\ar[d]^{g_0}\\
1 \ar[r]_{ \varepsilon}& \Omega_{t_0}
}
$$
Conversely, given $g_0$ making the last square a pullback in $\ck$, then the corresponding morphism $g\colon B\to (\Omega_t)_{t\in T}$ makes the above square a pullback. Therefore, $g_0$ is unique. This proves that $\varepsilon \colon 1 \to \Omega_{t_0}$ is a subobject classifier in $\ck$.
\end{proof}

\begin{corollary}\label{multi6} 
For a complete category $\ck$ with a strict initial object $\Sigma\ck$ is a topos if and only if $\ck$ is a topos.
\end{corollary}
This follows from the above Theorem and Remark \ref{multi1}.

However, we have seen in Example \ref{sigma4}(i) that for the topos $\Set_f$ the category $\Sigma\Set_f$ is not even cartesian closed.
 
\section{$\cp\ck$ is often locally cartesian closed}
We now turn to the free completion $\cp\ck$ of a category $\ck$ under colimits. In the present section we concentrate on the question whether $\cp\ck$ is (locally) cartesian closed, in the next ones we ask whether it is a topos and study the (co)wellpoweredness of it.

Recall the concept of a free completion
$$
\cp \ck
$$
under colimits: the category $\cp \ck$ is cocomplete, and there is a full embedding $E\colon \ck \to \cp \ck$ such that every functor $F\colon \ck \to \cl$ with $\cl$ cocomplete has an extension to a colimit-preserving functor $\overline F \colon \cp \ck \to \cl$, unique up to natural isomorphism. 

\begin{example}\label{completion}
(a) $\Set$ is the free colimit completion of the terminal category.

(b) $\Set^\to$ is the free colimit completion of the two-element chain.

(c) For small categories $\cp \ck$ is the presheaf category $[\ck^{\op},\Set]$ and $E$ is the Yoneda embeddding.

(d) For general categories $\cp\ck$ can be described as the full subcategory of $[\ck^{\op}, \Set]$ on all \textit{small functors}, i.e., small colimits of hom-functors (see \cite{U} 2.29). And $E$ is the codomain restriction of the Yoneda embedding, we denote it by 
$Y:\ck\to\cp\ck$.

(e) $\cp(\Set^{\op})$ is the category of all accessible set functors. Indeed, for every locally presentable category $\cl$ a functor into $\Set$ is accessible iff it is a small presheaf on $\cl^{\op}$ (see \cite{DL}). Thus $\cp(\cl^{op})$ is the category of all accessible set-valued functors on $\cl$.
\end{example}

The completions $\cp\ck$ and $\Sigma\ck$ are closely related: the former one is an exact completion of the latter. Recall that a category is \textit{exact} if it has finite limits, regular factorizations with regular epimorphisms stable under pullbacks, and effective equivalence relations. The \textit{exact completion} of a finitely complete category $\ck$, see \cite{CC}, is an exact category
$$
\ck_{\ex}
$$
with a full embedding $E:
\ck\to\ck_{\ex}$ such that every functor $F:\ck\to\cl$ with $\cl$ exact has an extension to an exact functor $\bar{F}:\ck_{\ex}\to\cl$ unique up to natural isomorphism. The following is Lemma 3 in \cite{R}.

\begin{lemma}\label{completion1}
For every finitely complete category $\ck$ we have
$$
\cp\ck=(\Sigma\ck)_{\ex}.
$$
\end{lemma}

Freyd introduced the following concepts in \cite{F}:

\begin{definition}\label{pre1}
(1) By a \textit{pre-limit} of a diagram $D$ is meant a set of cones such that every cone of $D$ factorizes through some of them
(not uniquely in general). A category is \textit{pre-complete} if every diagram has a pre-limit. Dual concept: pre-colimit and pre-cocomplete category.

(2)  A functor $F:\ck^{\op}\to\Set$ is \textit{petty} if it is a quotient of a coproduct of representable functors. It is \textit{lucid}
if, moreover, for every pair $u_1,u_2:P\to F$ of natural transformations the equalizer of $u_1,u_2$ has a petty domain.
\end{definition}

\begin{remark}\label{pre0}
(a) In the definition of a lucid functor we can restrict $P$ to representable functors. Thus, in categories where all subfunctors of representable functors are petty, any petty functor is lucid.

(b) A functor $F:\ck^{\op}\to\Set$ is petty iff its category $\el F$ of elements has a weakly initial set. That is, we have a set of objects $P_t$ $(t\in T)$ of 
$\ck$ and elements $p_t\in FP_t$ such that every element $p\in FP$ has, for some $t\in T$ and some morphism $f:P\to P_t$ of $\ck$, the form 
$p=F(f)(p_t)$.
\end{remark}

\begin{exs}\label{pre2}
(a) Complete categories are pre-complete.

(b) Small categories are pre-complete.

(c) Accessible categories are pre-cocomplete. Indeed, if $D$ is a diagram in a $\lambda$-accessible category and if $\lambda$ is chosen so that

(i) every object of $D$ is $\lambda$-presentable, and

(ii) $D$ has less than $\lambda$ morphisms,

\noindent
then the (essentially small) set of all cocones with $\lambda$-presentable codomains is a pre-colimit of $D$.
\end{exs}

The following is essentially Lemma 1 in \cite{R}.

\begin{lemma}\label{pre3}
Every lucid functor is small. If $\ck$  has finite pre-limits, then every small functor $\ck^{\op}\to\Set$ is lucid. 
\end{lemma}

\begin{proof}
(1) Every lucid functor $F$ is small. Indeed, since $F$ is petty, we have a quotient $\xymatrix@1{\gamma \colon \coprod\limits_{j\in J} \ck (-, B_j) \ar @{->>} [r] & F}$, $J$~a set. Thus $\gamma$ is the coequalizer of its kernel pair 
$$
\alpha,\beta\colon G \longrightarrow \coprod\limits_{j\in J}  \ck (-, B_j)
$$ 
which yields a petty functor $G$ (since $G$ is the equalizer of $\gamma\alpha$ and $\gamma\beta$ and $F$ is petty). Therefore, $G$ is also a quotient $\xymatrix@1{\varrho \colon\coprod\limits_{i\in I} \ck (-, A_i) \ar@{->>}[r]& G}$, $I$ a set. Consequently, $\gamma$ is the coequalizer of $\alpha \varrho$ and $\beta\varrho$, hence $F$ is small.

(2) If $\ck$  has finite pre-limits, every small functor $F:\ck^{\op}\to\Set$ is lucid. Indeed, by the usual reduction of colimits to coproducts and  coequalizers we have a coequalizer in $[\ck^{\op}, \Set]$ as follows
$$
\xymatrix@1{
\coprod\limits_{i\in I} \ck (-, A_i) \ar@<1ex> [r]^{\alpha} \ar@<-1ex> [r]_{\beta}& \coprod\limits_{j\in J} \ck (-, B_j) \ar [r]^{\quad\quad\gamma}& F
}
$$
where $I$ and $J$ are sets. By 1.4 and 1.2 in \cite{F}, both of the coproducts above are lucid, and by 1.9 in \cite{F}, applied to $\ck^{\op}$, it follows that $F$ is lucid.
\end{proof}

\begin{remark}\label{lim}
Whenever a limit of a  diagram in $\cp \ck$  exists, it is formed objectwise. That is, $\cp \ck$ is closed under existing limits in $[\ck^{\op}, \Set]$. This follows easily from Yoneda Lemma. 
\end{remark}

\begin{corollary}\label{lim1} $\cp \ck$ is (finitely) complete iff $\ck$ is (finitely) pre-complete.
\end{corollary}
\begin{proof}
Indeed, from (finite) pre-completeness of $\ck$ it follows that (finite) limits of lucid functors are lucid, see \cite[(1.7) 1.12]{F}. Thus  \ref{pre3} implies that $\cp \ck$ is closed under  (finite) limits in $[\ck^{\op}, \Set]$.

Conversely, let $\cp \ck$ be (finitely) complete. For every (finite) diagram $D$ in $\ck$ its composite with  the Yoneda embedding 
$Y\colon \ck \to [\ck^{\op}, \Set]$ has, by \ref{lim}, a petty pointwise limit in $[\ck^{\op}, \Set]$. This is clearly equivalent to $D$  having a pre-limit in $\ck$.
\end{proof} 
 
\begin{remark}\label{prod} 
Analogously, $\cp\ck$ has finite products iff $\ck$ has finite pre-products.
\end{remark}

\begin{lemma}\label{prelimit}
Given objects $A_{ij}$ $(i,j\in I)$ of $\ck$ for which $\cp\ck$ has the following product
$$
\underset{i\in I}\prod\coprod\limits_{j\in I} YA_{ij},
$$
it follows that a pre-product of $A_{ii}$ $(i\in I)$ exists in $\ck$.

More generally, given a diagram $D:\cd\times\cd\to\ck$ for which 
$$\lim_d\underset{d'}\colim YD(d,d')
$$ 
exists in $\cp\ck$, then $\ck$ has a pre-limit of the diagram $D\Delta_D$.
\end{lemma}
\begin{proof}
We prove the first statement, the latter is analogous. The functor $H=\prod\limits_i\coprod\limits_j YA_{ij}$ assigns to every object $P$ the collection of all $p=(p_i)_{i\in I}$ where $p_i:P\to A_{ij_i}$ is a morphism of $\ck$ for some $j_i\in I$. This follows from $\cp\ck$ being closed under colimits and existing limits in $[\ck^{\op},\Set]$. Since $H$ is small, it is petty.  Thus, by Remark \ref{pre0}(b)
we have a set of elements $p^t=(p^t_i)$ of $HP^t$ indexed by $t\in T$ where $p^t_i:P^t\to A_{ij_i}$ are morphisms for $i\in I$, with the following property:
for every $p\in HP$ there exists $t\in T$ and $f:P\to P^t$ such that $j_i=j^t_i$ and the triangles below commute
$$
\xymatrix@C=4pc@R=4pc{
P \ar[d]_{f} \ar [r]^{p_i} & A_{ij_i}\\
P^t \ar[ur]_{p^t_i} 
}
$$
for every $i\in I$. Consequently, if $T_0\subseteq T$ is the set of all $t\in T$ with $j^t_i=i$ for every $i\in I$, then the cones
$$
p^t_i:P^t\to A_{ii}\quad\quad (i\in I)
$$
for $t\in T_0$ form the desired pre-product.
\end{proof}
 
\begin{definition}[see \cite{R}]\label{pre-cc}
A category $\ck$ is called \textit{cartesian pre-closed} if it has finite products and for every pair of objects $A$ and $B$ 
the functor $\ck(A\times-,B):\ck^{\op}\to\Set$ is small. 
\end{definition}

The second author proved in \cite{R} Proposition 1 that for pre-complete categories $\ck$ cartesian closedness of $\cp\ck$ is equivalent to cartesian pre-closedness of $\ck$. We will show that this also implies pre-completeness of $\ck$:

\begin{theorem}\label{pre-cc1} 
Let $\ck$ have finite products. Then $\cp\ck$ is cartesian closed if and only if $\ck$ is pre-complete and cartesian pre-closed.
\end{theorem}
\begin{proof}
Let $\cp\ck$ be cartesian closed. Then $\ck$ is pre-complete. Indeed, for every diagram $D_0:\cd\to\ck$ we form the constant diagram $D_1:\cd\to\ck$ with value $1$, a terminal object, and use the exponential object $[H_1,H_0]$ where $H_i=\colim YD_i$. The existence of $[H_1,H_0]$ implies (due to colimits preserved by functor $K\times -$ in $\cp\ck$) that the following limit
$$
[H_1,H_0]=\lim\limits_{d\in\cd} [1,H_0]\cong\underset{d\in\cd}\lim\ \underset{d'\in\cd}\colim YD_0d'
$$
exists in $\cp\ck$. Apply Lemma \ref{prelimit} to
$$
D:\cd\times\cd\to\ck,\quad\quad (d,d')\mapsto D_0d'.
$$
We deduce that a pre-limit of  $D\cdot\Delta=D_0$ exists.

The rest follows from \cite{R} Proposition 1.
\end{proof}

\begin{remark}\label{open} 
In a cartesian pre-closed category every pair $A$ and $B$ of objects posseses a set of morphisms
$$
\ev_i \colon A \times [A,B]_i \to B \quad (i\in I)
$$
with the universal property of Definition \ref{multi} except that neither $i$ nor $\bar c$ are required to be unique. Indeed, by Remark
\ref{pre0}(b) this is precisely the fact that the functor $\ck(A\times-,B):\ck^{\op}\to\Set$ is petty. It is an open question whether (in analogy to the free coproduct completion) the pettiness of the above functors implies that they are small, i.e., that our category is cartesian pre-closed.
\end{remark}
	
\begin{example}\label{pre-cc2}
Every cartesian multi-closed category is cartesian pre-closed. Indeed, given the multi-universal morphism 
$\ev_i:A\times [A,B]_i\to B$ $(i\in I)$ of Definition \ref{multi}, then in $[\ck^{\op},\Set]$ we have
$$
\ck(A\times-,B) = \coprod\limits_{i\in I} \ck(-,[A,B]_i)
$$
This is clear since to give a morphism $c:A\times C\to B$ means precisely to give a (unique) $i\in I$ and a (unique) moprhism 
$\bar{c}:C\to [A,B]_i$. Thus, $\ck(A\times-,B)$ is small.
\end{example}



\begin{exs}\label{nonmcc}
We give examples of cartesian pre-closed categories which are not cartesian multi-closed.

(a) The category $\Top$ is not cartesian multi-closed (by Remark \ref{multi1}). It is cartesian pre-closed because it is infinitary extensive, thus Theorems 1 and 2 of \cite{R} apply.

(b) $\Top_0$, the full subcategory of $T_0$ spaces. The argument is the same.

(c) The full subcategory $\Set'$ of $\Set$ on all sets of cardinality other than continuum is cartesian pre-closed but not multi-closed. Indeed, this category has finite products. To prove that $\Set'(A\times -,B)$ is small, it is by Remark \ref{pre0} sufficient to verify that it is petty because subfunctors of representable presheaves on $\Set'$ are small (which is easy to verify, or see the proof of
Lemma \ref{cowell1} below). We present a weakly initial object in the category of elements of $\Set'(A\times -,B)$: Choose a set $T$ of cardinality larger than continuum. Then the set $[A,B]+T$ together with an arbitrary function 
$\ev:A\times ([A,B]+T)\cong A\times [A,B]+A\times T\to B$ whose left-hand component is the evaluation provides a weakly initial element of $\Set'(A\times -,B)$. In fact, given $f: A\times X\to B$ in $\Set'$ there exists a (not necessarily unique) $g:X\to [A,B]+T$ with $f = \ev.(A\times g)$. 

$\Set'$ is not cartesian multi-closed because, having a strictly initial object, it would be cartesian closed, see Remark \ref{multi1}. However, the sets $\Bbb N$ and $2$ obviously fail to have an exponential object in $\Set'$. (Since given $[\Bbb N,2]$ in $\Set'$, its
elements would bijectively correspond to maps $1\times \Bbb N\to 2$, hence $[\Bbb N,2]$ would have cardinality continuum.)
\end{exs}
We  are going to improve the above Theorem by moving to local cartesian closedness. For that we use the characterization of locally cartesian closed exact completions $\cc_{\ex}$ due to Carboni and Rosolini \cite{CR}, which we now recall. In that paper $\cc$ was supposed to have weak finite limits. Recently, Emmenegger showed that the result is only correct if finite completeness is assumed, see \cite{E}. The following Definition and Theorem are from \cite{CR}.

\begin{definition}\label{simple}
By a \textit{weak simple product} of a morphism $b:B\to A\times X$ over $A$ is meant an object $P$ and morphisms 
$\varepsilon:A\times P\to B$ and $w:P\to X$ forming a commutative triangle as follows
$$
\xymatrix@C=4pc@R=4pc{
A\times  P \ar[d]_{\varepsilon} \ar [r]^{A\times w} & A\times X\\
B \ar[ur]_{b}&
}
$$
They are also requested to have the following weak universal property: given morphisms $\varepsilon':A\times P'\to B$ and $w':P'\to X$
with $b\varepsilon'=A\times w'$, there exists a morphism $f:P'\to P$ such that $\varepsilon' = \varepsilon (A\times f)$ and $w'= wf$.
\end{definition}

\begin{theorem}\label{simple2}
A finitely complete category has a locally cartesian closed exact completion if and only if it has weak simple products.
\end{theorem}

\begin{remark}\label{simple1}
Weak simple products are also callled weak dependent products. The "pre" version was called approximate dependent product  in \cite{R}:
it is a set of objects $P_i$ $(i\in I)$ and morphisms $\varepsilon_i:A\times P_i\to B$, $w_i:P_i\to X$ $(i\in I)$ such that 
$b\varepsilon_i=A\times w_i$. The weak universal property states that given $\varepsilon':A\times P'\to B$ and $w':P'\to X$
with $b\varepsilon'=A\times w'$, there exists $i\in I$ and a morphism $f:P'\to P_i$ with $\varepsilon'=\varepsilon_i(A\times f)$ and $w'=w_if$.
\end{remark}

\begin{theorem}\label{lcc}
For a complete category $\ck$ the following conditions are equivalent:

(i) $\cp\ck$ is locally cartesian closed,

(ii) $\cp\ck$ is cartesian closed,

(iii) $\Sigma\ck$ is cartesian pre-closed, 

(iv) $\Sigma \ck$ has weak simple products, and

(v) $\ck$ is cartesian pre-closed.
\end{theorem}
\begin{proof}
iii $\to$ iv. In Lemma A1 of \cite{R} it is proved that a complete, cartesian pre-closed category $\cc$ has approximate dependent products. We apply this to $\Sigma\ck$ to conclude that $\Sigma\ck$ has weak simple products. Indeed, given a morphism $b:B\to A\times X$, let
$\varepsilon_i:A\times P_i\to B$ and $w_i:P_i\to B$ $(i\in I)$ form an approximate dependent product in $\Sigma\ck$. Since $A\times -$
preserves coproducts (by Example \ref{sigma0}(e) and Remark \ref{sigma}(2)), the following morphisms $[\varepsilon_i]_{i\in I}:A\times \coprod P_i\to B$
and $[w_i]_{i\in I}:\coprod P_i\to A$ are easily seen to form a weak simple product of $b$ over $A$.

\noindent
iv $\to$ i. $\cp\ck=(\Sigma\ck)_{\ex}$ is locally cartesian closed by Theorem \ref{simple2}.

\noindent
i $\to$ ii is clear.

\noindent
ii $\to$ iii. In Theorem 1 of \cite{R} every complete, infinitary extensive category $\cc$ with $\cc_{\ex}$ cartesian closed is proved to be cartesian pre-closed. Apply this to $\cc=\Sigma\ck$, using Theorem \ref{pre-cc1}.

\noindent
v $\leftrightarrow$ ii follows from Theorem \ref{pre-cc1}.
\end{proof}
 
\begin{corollary}\label{pre-cc3}
Let $\ck$ be a complete category which is
\begin{enumerate}
\item[a.] cartesian closed, or
\item[b.] additive, or
\item[c.] dual to an extensive category.
\end{enumerate}
Then $\cp\ck$ is locally cartesian closed.
\end{corollary}

Thus categories $\cp\Set$, $\cp\Ab$, $\cp\Top$ and $\cp(\Set^{\op})$ (= accessible set functors) are locally cartesian closed.

\section{$\cp\ck$ is a pretopos but seldom a topos} 
Recall that a \textit{pretopos} is an exact and extensive category. That is, a category which is regular, has effective equivalence relations, and has finite coproducts which are disjoint and universal.
 
\begin{lemma}\label{pretopos}
Let $\ck$ be a finitely pre-complete category. Then $\cp\ck$ is a pretopos.  
\end{lemma}
\begin{proof}
By Remark \ref{lim} and Corollary \ref{lim1}, $\cp\ck$ is closed under existing limits in $[\ck^{\op},\Set]$, and it is obviously closed under colimits. Thus $\cp\ck$ shares with $\Set$ all the exactness properties of colimits and finite limits. In particular, $\cp\ck$ is exact, extensive, thus, a pretopos.
\end{proof}
 
\begin{remark}\label{infinitary}
By the same argument we derive that if $\ck$ is finitely pre-complete, then

(a) $\cp\ck$ is infinitary extensive, see Remark \ref{extensive0},

(b) every monomorphism in $\cp\ck$ is regular, and

(c) every epimorphism in $\cp\ck$ is regular.

\noindent
For small categories $\ck$, we know that $\cp\ck=\Set^{\ck^{\op}}$ is a topos. The converse "almost" holds:
\end{remark}

\begin{theorem}\label{topos}
Given a category $\ck$  with copowers and finite intersections, then $\cp\ck$ is a topos iff $\ck$ is essentially small.
\end{theorem}

\begin{proof}
(1) We consider first the case that $\ck$ is a preordered class. We are to prove that if $\cp\ck$ is a topos, then $\ck$ is essentially small. Since $\cp\ck$ has a terminal object, $\ck$ has a pre-terminal set of objects, and since $\cp\ck$ is wellpowered, so is $\ck$. It is easy to see that a wellpowered preordered class with a pre-terminal set of objects is essentially small.

(2)  Next consider $\ck$ having an object $L$ which is the domain of a parallel pair of distinct morphisms. Since for every cardinal $\alpha$ the coproduct injections of $\coprod\limits_{\alpha} L$ are then pairwise distinct, $\ck$ is not essentially small. We are going to prove that $\cp\ck$ is not topos.

Assuming that $\cp\ck$ has a subobject classifier $\Omega\colon \ck^{\op}\to \Set$, we derive a contradiction. Since $\cp\ck$ contains all hom-functors, Yoneda lemma implies that  (up to natural isomorphism) $\Omega$ is given as follows: on objects $A$ we have 
 $$
 \Omega A = \mbox{\ all small subfunctors of \ $\ck(-, A),$}
 $$
 and on morphisms $h\colon A\to B$ the map $\Omega h$ forms preimages of subfunctors under $\ck(-, h)$. We are going to prove that this presheaf $\Omega$ is not petty, which is the desired contradiction.
 
 For every set of elements of $\Omega$,
 $$
\sigma_i \colon S_i \rightarrowtail \ck (-, A_i)
\qquad (i\in I)
$$
we are going to find an object $K$ of $\ck$ and an element of $\Omega$
$$
\xymatrix@1{ 
\tau
\colon T \rightarrowtail\ck (-, K)
}
$$
such that $T\ne \Omega h(S_i)$  for all $i\in I$ and all $h\colon K\to A_i$. Our choice of $K$ is a copower of the above object $L$
$$
K=\coprod_{\alpha} L
$$
where $\alpha$ is a cardinal with
$$
\card \ck (L, A_i) < \alpha \quad \mbox{for all \ $i\in I$}\,.
$$
Denote by $v_t \colon L\to K$ \ ($t<\alpha$) the coproduct injections. Choose arbitrary two elements $t\ne t'$ in $\alpha$ and form the  intersection $m= v_t\cap v_{t'}$   
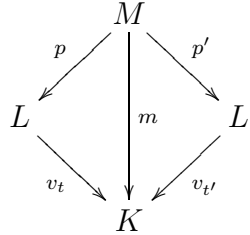
$$
\xymatrix{
& M \ar[dl]_{p} \ar[dd]^{m} \ar [dr]^{p'} &\\
L\ar[dr]_{v_t} && L \ar[dl]^{v_{t'}}\\
& K &
}
$$
Given any distinct elements  $s\ne s'$ in $\alpha$, they yield the same intersection:
$$
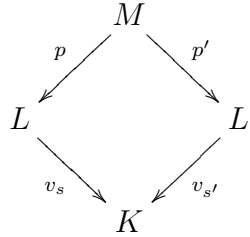
\xymatrix{
& M \ar[dl]_{p}   \ar [dr]^{p'} &\\
L\ar[dr]_{v_s} && L \ar[dl]^{v_{s'}}\\
& K &
}
$$
This follows from the (obvious) fact that there is an isomorphism $\sigma \colon K\to K$ with $v_s =\sigma \cdot v_t$ and $v_{s'} = \sigma \cdot v_{t'}$.

Our choice of the subfunctor $T$ of $\ck(-, K)$ uses the above morphism $p\colon M \to L$: put $L_s =L$ and  $p_s =p$ for all $s <\alpha$ and define a small presheaf $T$ as the wide pushout of the following morphisms
$$
\ck (-, p_s) \colon \ck(-, M) \to \ck (-, L_s) \qquad (s<\alpha)\,.
$$
Up to natural isomorphism, $T$ can be described as follows: on objects $X$ put
$$
TX =\coprod_{s<\alpha} \ck(X, L_s) \big/
\sim_X
$$
where $\sim_X$ is the least equivalence relation merging, for every morphism $f\colon X\to M$, all  $p_s \cdot f \colon X \to L_s$ with $s<\alpha$ to one element (denoted by $\langle f\rangle$).
The equivalence class of a morphism $g\colon X\to L_s$  not factorizing through $p$ is a singleton, and we denote it by $g_s$.

This small presheaf $T$ is a subobject of $\ck (-, K)$ since the cocone $\ck (-, v_s) \colon \ck(-,L)\to\ck (-, K)$ fulfills 
$\ck(-, v_s) \cdot \ck(-, p_s) = \ck (-, m)$ (independent of $s<\alpha$). This yields a natural transformation
$$
\tau \colon T\to \ck (-, K) \quad \mbox{with}\quad \tau \cdot \ck(-, p_s) = \ck(-, m) \quad (s<\alpha)\,.
$$
It assigns to $\langle f\rangle$ for $f\colon X\to M$ the value
$$
\tau_X (\langle f\rangle) = m\cdot f
$$
and to $g\colon X \to L_s$ not factorizing through $p$ the value
$$
\tau_X (g_s) =v_s\cdot g\,.
$$
This morphism $\tau$ is monic:
\begin{enumerate}
\item[(a)]
 Given distinct $f$, $f'\colon X \to M$, then $m\cdot f\ne m\cdot f'$ since $m$ is monic: indeed, each $v_t$ is a split monomorphism (the codiagonal splitting it).
\item[(b)] 
Given distinct $g\colon X \to L_s$ and $g'\colon X \to L_{s'}$ not factorizing through $p$, then in case $s=s'$ we have 
$v_s \cdot g \ne v_s \cdot g'$ since $v_s$ is a (split) monomorphism, and in case $s\ne s'$ we have $v_s \cdot g\ne v_{s'} \cdot g'$:
Otherwise $g$ would factorize through $p$ (and $g'$ through $p'$) due to the above pullback $v_s\cdot p= v_{s'} \cdot p'$.
\item[(c)]
Given $f\colon X \to M$ and $g\colon X\to L_s$ not factorizing through $m$, then $m\cdot f \ne v_s \cdot g$: since $v_s$ is a split monic and $m=v_s\cdot p$, the equality $m\cdot f = v_s\cdot g$ would yield a factorization $f$ or $g$ through~$p$.
\end{enumerate}

It remains to prove
$$
T\neq \Omega h(S_i) \quad \mbox{for all}\quad i\in I,\ \ h\colon K\to  A_i\,.
$$
Assuming to the contrary $T= \Omega h(S_i)$, we prove that the morphisms
$$
h\cdot v_t \colon L\to A_i \qquad (t<\alpha)
$$
are pairwise distinct, contradicting $\card \ck(L, A_i)<\alpha$.

For elements $t$, $t'$ of $\alpha$  with $h\cdot v_t = h\cdot v_{t'}$ we are going to verify that $t=t'$. The equality $T=\Omega h(S_i)$ means that in $[\ck^{\op}, \Set]$ we have  a pullback as follows:
$$
\xymatrix@C=3pc{
T \ar [r]^{\tau\quad } \ar[d]_{\varrho} & \ck(-, K) \ar[d]^{h\cdot (-)}\\
S_i \ar[r]_{\sigma_i\quad } & \ck(-, A_i)
}
$$
 Apply this to $\id_L$ in the component of $TL$ indexed by $t$: we get
 $$
 x= \varrho_L \big( (\id_L)_t\big) \in S_i L
 $$
 with 
 $$
 (\sigma_i)_L (x) = h\cdot v_t
 $$
 since $\tau_L((\id_L)_t)= v_t$. (Observe that $\id_L$  clearly does not factorize through $p$.) The codiagonal $\nabla \colon L+L\to L$ yields
$$ y = S_i \nabla (x) \in S_i (L+L)
$$
with
$$
(\sigma_i)_{L+L} (y) = h\cdot [v_t, v_{t'}]\,.
$$
Indeed, naturality of $\sigma_i$ gives, since $h\cdot v_t = h\cdot v_{t'}$
$$
(\sigma_i)_{L+L} S_i\nabla (x) = \ck (-, \nabla)(\sigma_i)_L (x) = h\cdot v_t \cdot \nabla = [h \cdot v_t, h\cdot v_{t'}].
$$
Thus, the above pullback applied to $L+L$ provides 
$$
z\in T(L+L) \quad \mbox{with}\quad \tau_{L+L} (z) = [v_t, v_{t'}]\,.
$$
Since $v_t$ does not factorize through $m$, neither does $[ v_t, v_{t'}]$, therefore $z$ lies in some component $s<\alpha$ of $T(L+L)$ and $v_s\cdot z = [v_t, v_{t'}]$. Thus $v_t$ factorizes through $v_s$, proving $t=s$, and $v_{t'}$ also factorizes through $v_s$,  hence $v_{t'} = s= v_t$, as desired.
\end{proof}

\begin{example}\label{topos1}
The categories $\cp\Set$ and $\cp\Cat$ are locally cartesian closed (by Theorem \ref{lcc}) but not toposes.
\end{example}

\begin{problem}\label{opb}
{
\em
Is there a category $\ck$ which is not essentially small with $\cp\ck$ a topos?
}
\end{problem}

\begin{remark}\label{generic}
(1) Menni characterized categories whose exact completion is a topos, see \cite{M}. He defined a \textit{generic proof} as
a map $\theta : \Theta \rightarrow \Lambda$ such that for every map $f:Y\rightarrow X$ there exists a map $g: X \rightarrow \Lambda$ for
which $f$ factorizes through ${g^*}\theta$, and ${g^*}\theta$ factorizes through $f$. And he proved (in Theorem 1.2) that a category has weak simple products and a generic proof iff its exact completion is a topos.

(2) Consequently, the theorem just proved implies for complete, cartesian pre-preclosed categories $\ck$ that the free colimit completion $\cp\ck$ is a topos iff the free coproduct completion $\Sigma\ck$ has a generic proof.
\end{remark}

\begin{corollary}\label{generic1} If a complete, cartesian pre-closed category with copowers is not essentially small, then its free coproduct completion does not have a generic proof.
\end{corollary}
See Theorem \ref{lcc}.

The following examples are proved by Menni in a somewhat more technical manner (see \cite{M}, Propositions 5.7 and Lemma 5.5):

\begin{example}\label{generic2}
 (1) The category $\Set^{\rightarrow}$ does not have a generic proof. Indeed, it is equivalent to $\Sigma \Set$ (see Example \ref{sigma0}(c)).

(2) For every small category $\ck$ the free coproduct completion has a generic proof. The exact completion, which is the presheaf category for $\ck$, is namely a topos.
\end{example}

\section{$\cp \ck$ is sometimes wellpowered but not often}
This last section is devoted to the question whether $\cp\ck$ is wellpowered or cowellpowered. We first observe that the two problems are more or less equivalent.

\begin{prop}\label{well}
Every cowellpowered category $\cp \ck$ is also wellpowered. The converse holds for finitely pre-complete categories $\ck$.
\end{prop}

\begin{proof}
(1) Let $\cp \ck$ be cowellpowered. Following Remark \ref{infinitary}, every monomorphism $\mu \colon F\to G$ in $\cp \ck$
is monic in $[\ck^{\op}, \Set]$, too. Form the pushout in $[\ck^{\op}, \Set]$:

 $$
\xymatrix{
& F \ar [dl]_{\mu} \ar [dr]^{\mu}& \\ 
 G \ar [dr]_{\alpha_1}^{\quad}& &G \ar [dl]^{\alpha_2}\\
 & H & }
 $$
Since $F$ and $G$ are small, so are $H$ and $G+G$. Thus we obtain a quotient $[\alpha_1, \alpha_2]\colon G+G \to H$ in $\cp \ck$.

Distinct subobjects of $G$ in $\cp\ck$ yield distinct quotients in $\cp \ck$. Indeed, $\Set$ certainly has the corresponding property, therefore, so does $[\ck^{\op}, \Set]$. Since $G+G$ has only a set of quotients, $G$ has only a set of subobjects.

(2) Let $\ck$ be finitely pre-complete, and $\cp \ck$ be wellpowered. Following Remark \ref{infinitary}, every epimorphism $\alpha:F\to G$ is regular in  $\cp \ck$. Hence it is determined by its kernel pair $\beta_1,\beta_2:H\to F$. Since $H$ is a subobject of $F\times F$,
$\cp \ck$ is cowellpowered.
\end{proof}

\begin{prop}\label{well1}
If a wellpowered category $\ck$ has all morphisms monic, then $\cp \ck$ is cowellpowered.
\end{prop}
\begin{proof}
 (1) For every functor $F\colon \ck^{\op}\to \Set$ denote by $M_F$ a class of objects representing all $X$ with $FX\ne \emptyset$ up to isomorphism. Then we verify that $F$ is small iff $M_F$ is a set.
 
Indeed, let $F$ be small. If $F= \ck(-, A)$, then $M_F$ is a set because $A$ has only a set of subobjects. If $F$ is a colimit of a diagram with objects $F_i$, $i\in I$, then clearly $M_F=\underset{i\in I}{\bigcup} M_{F_i}$. This is a set if $F$ is small.
 
Conversely, let $M_F$ be a set. Recall that $F$ is a colimit of the diagram of its elements. More precisely, we form the category $\el F$ of all pairs $(X,x)$ where $X$ is an object of $\ck$ and $x\in FX$. Morphisms $f\colon (X,x) \to (Y,y)$ are those morphisms $f\colon X\to Y$ of $\ck$ with $Ff(y)=x$.
 The diagram
 $$
 D\colon \el F \to [\ck^{\op}, \Set]\,, \qquad (X,x)\mapsto \ck (-, X)
 $$
 has canonical colimit $F$. Since $M_F$  is a set, $\el F$ is essentially a small category, thus, $F$ is small.
 
(2) The category $\ck$ clearly has non-empty pre-limits. (Indeed, choose an object $A$ of the given diagram $D$. Then every cone of $D$ has a subobject of $A$ as its domain.) Thus, as in Lemma \ref{pretopos}, any epimorphism $\alpha:F\to G$ is regular and thus it
is a coequalizer in $[\ck^{\op}, \Set]$, which clearly implies $M_F=M_G$. Therefore, every quotient of $F$ in $\cp \ck$ is determined by quotients of $FX$, $X\in M_F$ in $\Set$. This proves that all quotients form a set.
\end{proof}

\begin{prop}\label{cowell1}
$\cp\ck$ is cowellpowered for every finitely complete, wellpowered category $\ck$ such that
\begin{enumerate}
\item[(1)] all hom-sets of objects that are neither terminal nor initial are non-empty, and
\item[(2)] morphisms with non-terminal codomains have (split epi, mono)-factorizations.
\end{enumerate}
\end{prop}
\begin{proof}
(a) We first prove for every coproduct of representables
$$
F=\coprod\limits_{i\in I} \ck(-,A_i)
$$
that it has only a set of subobjects in $\cp\ck$. Since $\cp\ck$ is closed under coproducts in $[\ck^{\op},\Set]$ and $[\ck^{\op},\Set]$
is infinitary extensive, it is sufficient to prove that every representable functor has only a set of subobjects.

To give a subobject $G$ of $\ck(-,A)$ means to give, for every object $X$ of $\ck$, a subset $GX\subseteq \ck(X,A)$ such that for all morphisms
$f:Y\to X$ we have 
$$
u\in GX\text{ implies } uf\in GY.
$$
In case $A$ is a terminal object, then $G$ is determined by a class $\cg\subseteq\obj\ck$ such that no morphism with domain in $\cg$ has codomain in the complement $\obj\ck\setminus\cg$. Due to (1) either $\cg$ or its complement consists of objects isomorphic to $0$ (initial) or $1$ (terminal), thus, the number of such subfunctors $G$ is finite.

Suppose $A$ is non-terminal. We prove that $G$ is determined by its values at (i) 1 and (ii) subobjects of $A$. Since $\ck$ is wellpowered,
this proves that there is only a set of possibilities.

For that, consider an object $X$. Given $f\in GX$ factorize it as a split epic $e_f:X\to A'$ followed by a monic $m_f$.
If $r$ splits $e_f$ then $m_f=fr$ and, since $f\in GX$, we have $m_f\in GA'$. Conversely, since $f=m_fe_f$, $f\in GX$ whenever 
$m_f\in GA'$. This proves our claim that $G$ is determined by $G1$ and all $GA'$.

(b) The full subcategory of $[\ck^{\op},\Set]$ formed by all petty functors is cowellpowered. Indeeed, we only need to prove that coproducts of hom-functors $F=\coprod_{i\in I} \ck(-,A_i)$ have sets of quotients in $[\ck^{\op},\Set]$. Observe that $F\times F$ is also a coproduct of hom-functors $\coprod_{i,j\in I} \ck(-,A_i\times A_j)$, thus, by (a) it has only a set of quotients: see the proof of Proposition \ref{well}. 

(c) Following Remark \ref{pre0}, it is sufficient to prove that every subfunctor of a hom-functor is petty.  

1. Consider first $\ck(-,1)$ for the terminal object $1$. To give a subfunctor $M$ means to specify a class $\ca$ of objects such that 
$A\in\ca$ and $B\in\ck\setminus\ca$ implies $\ck(B,A)=\emptyset$; then the subfunctor assigns $\emptyset$ to objects of $\ca$ and $1$ else.  Due to (1), the only possibilities are $\ca=\emptyset$ or $\ca$ consisting of initial objects. In both cases every element $x\in MB$ 
yields a weakly initial object of $\el M$, thus $M$ is petty.

2. Next let $M$ be a subfunctor of $\ck(-,A)$ where $A$ is non-terminal. Without loss of generality assume that $MX\subseteq\ck(X,A)$ for all $X\in\ck$ and that on morphisms $f:X\to Y$ we have $(Mf)(u)=uf$. Since $\ck$ is wellpowered, we have a set $m_i:A_i\to A$ $(i\in I)$ of subobjects representing the images of all morphisms $f:X\to A$, $X\in\ck$, lying in $MX$. As above, each $m_i$ lies in $MA_i$. 
The elements $m_i$, $i\in I$, form a weakly initial set in $\el M$: for every $f\in MX$ as above we have $f=(Me)(m_i)$.  
\end{proof}

\begin{example}\label{cowell2}
$\cp\ck$ is wellpowered and cowellpowered for all of the following categories and their duals:

(1) $\Set$.

(2) Vector spaces over every field.

(3) Sets and partial functions.

\noindent
Indeed, $\ck$ as well as $\ck^{\op}$ satisfy the assumptions of the above Proposition. ($\Set^{\op}$ was the reason for the complicated condition (2).)
\end{example}

Recall that a category $\ck$ is \textit{concrete} it there exists a faithful functor $U:\ck\to\Set$. 
 
\begin{remark}\label{concrete}
(1) For a finitely pre-complete category $\ck$, the category $\cp\ck$ is concrete if and only if it is wellpowered.
 
Indeed, $\cp\ck$ is finitely complete (Corollary \ref{lim1}) and monomorphisms are regular (Remark (\ref{infinitary}).
Thus the result follows from \cite{F1} Theorem 4.1(iii). 

(2) The limit completion of a category $\ck$ is dual to $\cp(\ck^{\op})$. Thus the limit completion of the categories in the above example are wellpowered and cowellpowered.
\end{remark}
 
\begin{remark}\label{cowell3}
(1) The category
$$
\Acc[\Set,\Set]
$$
of all accessible set functors  is wellpowered and cowellpowered. Indeed, recall from Remark \ref{completion}(e) that
this is precisely the category $\cp(\Set^{\op})$. Recall also that this category is a locally cartesian closed pretopos (see Corollary \ref{pre-cc3} and Lemma \ref{pretopos}).

(2) This category has also been studied by Barto \cite{B} who has also proved that it is concrete, and \textit{universal}. That is,
every concrete category can be fully embedded into $\Acc[\Set,\Set]$. 
\end{remark}

Recall that a \textit{clique} on a set $X$ is a graph whose arrows lead from every node to all
distinct nodes.
Graph homomorphisms between cliques are precisely the monic maps. Therefore, if $A_k$ is a clique on
$\aleph_k$ nodes, then
\begin{enumerate}
	\item[(a)]for all ordinals $k$ infinitely many morphisms exist from $A_0$ to $A_k$, and
	\item[(b)]no morphism from $A_k$ to $A_i$ exists if $i<k$.
\end{enumerate}

Recall that a class $\crr$ of morphisms of a category $\ck$ is \textit{right-cancellative} if $uv\in\crr$ implies that $u\in\crr$.

\begin{definition}\label{clique}
	A class of objects $A_k$  for $k\in \Ord$ of a category is called \textit{clique-like} provided that
	there exists a right-cancellative\footnote{We only need a weaker condition: given $uv\in\crr$ and $v\in\crr$, then $u\in\crr$.}  
	class $\mathcal R$ of morphisms such that
	
	\begin{enumerate}
	\item[(a)] for every ordinal $k$ more than one morphism from $A_0$ to $A_k$ exists in $\mathcal R$, and
	\item[(b)] no morphism from $A_k$ to $A_i$ exists in $\mathcal R$ either
	
	(b1) for all pairs of ordinals $0<i<k$

or

(b2) for all pairs of ordinals $0<k<i$.

\end{enumerate}
\end{definition}

\begin{prop}\label{clique0}
	If a category $\ck$ has a clique-like class of objects, then $\cp(\ck^{\op})$ is not cowellpowered.
\end{prop}

\begin{proof}
	For every ordinal $0<k$ denote by
	$$
	\pi_k\colon\ck(A_0,-)\to P_k
	$$
	the multiple coequalizer of the following morphisms of $\cp(\ck^{\op})$ (a full subcategory of
	$\Set^{\ck}$):
	$$
	\ck(f,-)\colon
\ck(A_k,-) \to\ck(A_0,-)
	$$
	for all morphisms $f\colon A_0\to A_k$ of $\mathcal R$. These small functors $P_k$ are quotients of
	$\ck(A_0,-)$ in $\cp(\ck^{op})$.
	It remains to prove that they are pairwise non-isomorphic. Indeed, consider $0<i<k$.

Assume first that (b1) holds. Choose distinct morphisms $g,h\colon A_0 \to A_i$ in $\mathcal R$. They are merged by $(\pi_i)_{A_i}$ because
the element $\id_{A_i}$ of $\ck(A_i,A_i)$ is sent to $g$ by $\ck(g,-)$ and to $h$ by $\ck(h,-)$. But $(\pi_k)_{A_i}$ does not merge $g$ and $h$. This follows from the fact that no element $u$ of $\ck(A_k,A_i)$ lies in $\mathcal R$. Therefore, due to the right cancellativity,
$g$ cannot be a composition $uf$ for any $f:A_0\to A_k$ in $\crr$ and any $u:A_k\to A_i$. Thus $g$ cannot be merged with any distinct morphism from $A_0$ to $A_i$ by $(\pi_k)_{A_i}$.

In the case that (b2) holds the proof is the same, we just calculate the components $(\pi_i)_{A_k}$ and $(\pi_k)_{A_k}$.

\end{proof}

\begin{example}\label{clique1}
	In the category $\Set_m$ of sets and monomorphisms we have a clique-like collection: choose a set $A_k$
	of power $\aleph_k$ and $\crr =$ all morphisms. Thus $\cp(\Set_m^{op})$
	is not cowellpowered.
	Compare this with the fact that by Proposition \ref{well1} $\cp\Set_m$ is.
	
\end{example}

\begin{example}\label{clique2}
	For the following categories $\ck$ the colimit completions $\cp\ck$ and
	$\cp(\ck^{\op})$ are not cowellpowered: we present a clique-like class of objects in $\ck$ as well as
	in its dual. Here $\mathcal R$ is always chosen
	to be the class of all morphisms of $\ck$.
	
	(1) The category of graphs: let $A_k$ be a clique on $\aleph_k$ vertices.
	This is a clique-like class in $\Gra$. Now take the same class except changing $A_0$
	to the complete graph on 2 vertices (so that more than one homomorphism
	exists from every nonempty graph into $A_0$). The result is a clique-like
	class in $\Gra^{op}$.
	
	(2) Any \textit{alg-universal} category $\ck$, i.e, one admitting a full embedding $E\colon\Gra\to\ck$.
	The monograph \cite{PT} presents a number of examples of such categories, e.g., semigroups or rings
	with unit.
	
	If $A_k$ is
	a clique-like collection in $\Gra$ or $\Gra^{op}$, then $EA_k$ is one in $\ck$ or $\ck^{op}$, resp.

(3) The category $\cp\Set^{op}$ of accessible set functors is, as we have seen above, cowellpowered. By Remark \ref{cowell3}(2), 
this category is alg-universal, therefore $\cp(\cp\Set^{op})$ is not cowellpowered.

(4) The category $\Un$ of algebras on one unary operation. To present a clique-like class, we
first construct unary algebras $B_i$ for $i\in\Ord$ with no homomorphisms from $B_i$ to $B_j$ whenever
$i>j + \omega$ .

Recall from \cite{F} the concept of the \textit{rank} $r(x)$ of an element $x$ of an algebra $A$,
which is either an ordinal or $\top$, an element
larger than all ordinals. We work with $A$ given by a set $X$ and its
endomap $\alpha$. The rank is defined to be $r(x) =\top$ iff
there exist elements $x_n$ with $\alpha(x_{n+1})=x_n$
for all $n<\omega$ and $\alpha(x_0)=x$. For all other elements $x$ of $A$ rank is defined by
$$
r(x) =\sup \{ r(y); y\in X, \alpha (y)=x\}.
$$

Every homomorphism $h\colon (X, \alpha) \to (X', \alpha')$ is nondecreasing on ranks:
$$r( h(x)) \geq r(x)$$
for all $x\in X$. See \cite{F}, Proposition 5.2.

Let us construct algebras $B_i=(X_i, \alpha_i)$ for $i\in \Ord$ such that $B_i$ contains an element
$x_i$ of rank $i$, but no element
of  rank at least $ i + \omega$. We proceed by transfinite induction:

\begin{itemize}
	
	\item For $i=0$ put $X_0=\N$, $\alpha_0(n) = n+1$ and $x_0 =0$.
	
	\item Given $(X_i, \alpha_i, x_i)$ put $X_{i+1}= X_i$, $\alpha_{i+1}=\alpha_i$ and $x_{i+1}
	=\alpha_i(x_i)$.
	
	\item Given a limit ordinal $j$ put $(X_j, \alpha_j)=\coprod\limits_{i<j} (X_i, \alpha_i)/\approx$,
	where $\approx$ is the least congruence with all
	$x_i$, $i<j$, forming one class. Call that class $x_j$.
	
\end{itemize}

Now to obtain a clique-like collection $A_k$ of algebras in $\Un$, let $A_0$ be the free algebra on
one generator (in which all ranks are finite).
For every ordinal $k>0$ let $A_k$ be the above algebra $B_i$ for $i={\aleph_k}$ (which has an element
of rank ${\aleph_k}$ but none of rank ${\aleph_{k+1}}$).

Analogously a clique-like class in $\Un^{op}$ is given: just change $A_0$
to be the cofree unary algebra on a 2-element set X. (This can be described as the algebra $X^{\N}$
with the operation assigning to every map
$f\colon \N \to X$ the map $f(1+-)$, whereas the couniversal morphism is the evaluation at $0$.)
\end{example}

 \begin{remark}\label{clique4}
 	Let $\ck$ be a concrete category for which a faithful functor $U$ to the category of sets is given. Recall that $(\ck,U)$ is called
 	\textit{almost alg-universal} if there exists
 	an embedding $E\colon\Gra \to \ck$ which is \textit{almost full}. This means
 	that every morphism $f\colon EX \to EY$ of $\ck$ with $Uf$ non-constant has the
 	form $f=Eh$ for a unique graph homomorphism $h:X \to Y$.
 	
 	Given a clique-like collection $A_k$ in $\Gra$, it then follows
 	that $EA_k$ is clique-like in $\ck$: choose $\mathcal R$ to be the
 	class of all morphisms $f$ with $Uf$ nonconstant. Analogously,
 	every clique-like collection in $\Gra^{op}$ yields one in $\ck^{op}$.
\end{remark}

\begin{example}\label{clique5}
	For the following categories $\ck$ the colimit completions $\cp\ck$ and
	$\cp(\ck^{\op})$ are not cowellpowered: as proved in
	\cite{PT}, each of these categories is almost alg-universal.
	
	(1) Monoids and homomorphisms.
	
	(2) Posets and monotone functions.
	
	(3) Lattices and homomorphisms.
	
	(4) Topological spaces and continuous functions.
	
	(5) Small categories and functors. Here $U$ assigns to every small category the set of all morphisms.
	
	Under the assumption that no cardinal is measurable, further examples are
	
	(6) Compact Hausdorff spaces and continuous functions.
	
	(7) Metric spaces and (uniformly) continuous functions.
	
\end{example}

\begin{example}\label{Ab}
	For the category $\Ab$ of abelian groups we also have that
	$\cp\Ab$ and $\cp(\Ab^{op})$ are not cowellpowered.
	This follows from the result of Prze\'zdziecki \cite{Pr}
that there exists an embedding $E\colon\Gra \to \Ab$ such that for every
pair of graphs $X$ and $Y$ the homomorphism group $\Ab(EX,EY)$
is the free abelian group on $\Gra(X,Y)$. In particular, if $C_k$
is a clique on $\aleph_k$ vertices in $\Gra$, then the abelian groups
$A_k=EC_k$ fulfil for every ordinal $k$ that

(a) there exist infinitely many homomorphisms from $A_0$ to $A_k$, and

(b) if $k>i$, there exist no non-zero homomorphism from $A_k$ to $A_i$.

Thus, we get a clique-like class in $\Ab$ w.r.t. the collection $\mathcal R$ of zero homomorphisms.

Analogously, a clique-like class on $\Gra^{op}$ yields a clique-like class
in $\Ab^{op}$.

\end{example}

In all of the above concrete examples except $\Set_m$ the categories $\ck$ are finitely complete and cocomplete. Thus, the fact that the
colimit completions of those categories (and their duals) are not cowellpowered
implies that they are also not wellpowered, see Proposition 5.1.

\end{document}